\documentclass[a4paper,10pt]{article}
\usepackage[a4paper,left=3.0cm,right=3.0cm,top=3.5cm,bottom=3.0cm]{geometry}
\usepackage[footnotesize]{caption}
\usepackage{amsmath,amssymb,amsthm}
\usepackage{graphicx}
\usepackage[round,sort&compress]{natbib}
\bibpunct{[}{]}{,}{n}{,}{,}

\usepackage{setspace} 

\newtheorem{theorem}{Theorem} 
\newtheorem{lem}{Lemma}
\newtheorem{prop}{Proposition}
\newtheorem{cor}{Corollary}

\begin{document}

\title{Distributed delays stabilize negative feedback loops}
\author{Samuel Bernard\thanks{Universit\'e de Lyon; Universit\'e Lyon 1; INSA de Lyon, F-69621; Ecole Centrale de Lyon; CNRS, UMR5208, Institut Camille Jordan, 43 blvd du 11 novembre 1918, F-69622 Villeurbanne-Cedex, France, \tt{bernard@math.univ-lyon1.fr}}}
\maketitle

\begin{abstract}
Linear scalar differential equations with distributed delays appear in the study of the local stability of nonlinear differential equations with feedback, which are common in biology and physics. Negative feedback loops tend to promote oscillation around steady states, and their stability depends on the particular shape of the delay distribution. Since in applications the mean delay is often the only reliable information available about the distribution, it is desirable to find conditions for stability that are independent from the shape of the distribution. We show here that the linear equation with distributed delays is asymptotically stable if the associated differential equation with a discrete delay of the same mean is asymptotically stable. Therefore, distributed delays stabilize negative feedback loops.
\end{abstract}

\section{Introduction}
The delayed feedback system of the form 
\begin{align}\label{eq:nl}
\dot x = F\Bigl(x, \int_0^{\infty} [d\eta(\tau)] \cdot g\bigl(x(t-\tau),\tau\bigr) \Bigr),
\end{align}
is a model paradigm in biology and physics \citep{monk2003, atay2003, adimy2006, rateitschak2007, eurich2005, meyer2008}. The first argument is the instantaneous part and the second one, the delayed or retarded part, which forms a feedback loop.  The function $\eta$ is a cumulative distribution of delays and $F$ and $g$ are nonlinear functions satisfying $F(0,0)=0$ and $g(0,\tau)=0$. When $F:\text{R}^d \times \text{R}^{d \times d} \to \text{R}^d$ and $g:\text{R}^d \times \text{R} \to \text{R}^{d \times d}$ are smooth functions, the stability of $x=0$ is given by the linearized form,
\begin{align}
\dot x = -A x - \int_0^{\infty} [B(\tau) \cdot d \eta(\tau)] x(t-\tau).
\end{align}
The coefficients $A$ and $B(\tau) \in \text{R}^{d \times d}$ are the Jacobian matrices of the instantaneous and the delayed parts, $\eta: [0,\infty) \to \text{R}^{d \times d}$ is the distribution of delays and $(\cdot)$ is the pointwise matrix multiplication. In biological applications, discrete delays in the feedback loop are often used to account for the finite time required to perform essential steps before $x(t)$ is affected. This includes maturation and growth times needed to reach reproductive age in a population \citep{hutchinson1948,mackey1978}, signal propagation along neuronal axons \citep{campbell2007}, and post-translational protein modifications \citep{monk2003,bernard2006b}. Introduction of a discrete delay can generate complex dynamics, from limit cycles to chaos \citep{sriram2008}. Linear stability properties of scalar delayed equations are fairly well characterized. However, lumping intermediate steps into a delayed term can produce broad and atypical delay distributions, and it is not clear how that affects the stability compared to a discrete delay \citep{campbell2009}. 

Here, we study the stability of the zero solution of a scalar ($d=1$) differential equation with distributed delays,
\begin{align}\label{eq:x}
 \dot x & = - a x - b \int_0^{\infty} x(t-\tau) d\eta(\tau).
\end{align}
The solution $x(t) \in \text{R}$ is the deviation from the zero steady state of equation (\ref{eq:nl}). Coefficients $a=-\text{D}_1F(0,0) \in \text{R}$ and $b = -\text{D}_2F(0,0) \neq 0$, and the integral is taken in the Riemann-Stieltjes sense. We assume that $\eta$ is a cumulative probability distribution function, i.e. $\eta: \text{R} \to [0, 1]$ is nondecreasing, piecewise continuous to the left, $\eta(\tau)=0$ for $\tau<0$ and $\eta(+\infty)=1$. Additionally, we assume that there exists $\nu>0$ such that 
\begin{align}\label{eq:nu}
\int_0^{\infty} e^{\nu \tau} d\eta(\tau) < \infty.
\end{align}
This last condition implies that the mean delay value is finite,
\begin{align*}
E & = \int_0^{\infty} \tau d\eta(\tau) < \infty.
\end{align*}
The corresponding probability density function is $f(\tau)$ given by $d \eta(\tau) = f(\tau) d\tau$, where the derivative is taken in the generalized sense. The distribution can be continuous, discrete, or a mixture of continuous and discrete elements. When it is a single discrete delay (a Dirac mass), the asymptotic stability of the zero solution of equation (\ref{eq:x}) is fully determined by the following theorem, due to Hayes \citep{hayes1950},
\begin{theorem}\label{th:hayes}
Let $f(\tau) = \delta(\tau-E)$ a Dirac mass at $E$. The trivial solution of equation (\ref{eq:x}) is asymptotically stable if and only if $a > |b|$, or if $b>|a|$ and
\begin{equation}\label{eq:Emax}
E <  \frac{\arccos(-a/b)}{\sqrt{b^2-a^2}}.
\end{equation}
\end{theorem}
There is a Hopf point if the characteristic equation of equation (\ref{eq:x}) has a pair of imaginary roots and all other roots are negative. For a discrete delay, the Hopf point occurs when equality in (\ref{eq:Emax}) is satisfied. Moreover, for any distribution $\eta$, there is a zero root along the line $-a=b$. At $-a=b=1/E$, there is a double zero root. When $a>-1/E$, all other roots have negative real parts, but when $a<-1/E$, there is one positive real root. Thus, the stability depends on $\eta$ if and only if $b>|a|$. Moreover, only a Hopf point can occur when $b>|a|$. Therefore, a distribution of delays can only destabilize equation (\ref{eq:x}) through a Hopf point, and only when $b>|a|$. This is a common situation when the feedback acts negatively on the system ($\text{D}_2F(0,0)<0$) to cause oscillations.

Assuming $b>0$ and making the change of timescale $t \to bt$, we have $a \to a/b$, $b \to 1$ and $\eta(\tau) \to \eta(b\tau)$. Equation (\ref{eq:x}) can be rewritten as
\begin{align}\label{eq:xx}
 \dot x & = - a x - \int_0^{\infty} x(t-\tau) d\eta(\tau).
\end{align}
The aim of this paper is to study the effect of delay distributions on the stability of the trivial solution of equation (\ref{eq:xx}), therefore, we focus on the region $|a|<1$. To emphasize the relation between the stability and the delay distribution, we will say that $\eta$ (or $f$) is stable if the trivial solution of equation (\ref{eq:xx}) is stable, and that $\eta$ (or $f$) is unstable if the trivial solution is unstable. 

It has been conjectured that among distributions with a given mean $E$, the discrete delay is the least stable one \citep{bernard01,atay2008}. If this were true, according to Theorem \ref{th:hayes}, all distributions would be stable provided that
\begin{equation}\label{eq:sc}
 E < \frac{\arccos(-a)}{\sqrt{1-a^2}}. 
\end{equation}
This conjecture has been proved for $a=0$ using Lyapunov-Razumikhin functions \citep{krisztin1990}, and for distributions that are symmetric about their means [$f(E-\tau) = f(E+\tau)$] \citep{miyazaki1997,bernard01,atay2008,kiss2009}. It has been observed that in general, a greater relative variance provides a greater stability, a property linked to geometrical features of the delay distribution \citep{anderson1991}. There are, however, counter-examples to this principle, and there is no proof that for $a \neq 0$ the least stable distribution is the single discrete delay. It is possible to lump the non-delayed term into the delay distribution using the condition found in \citep{krisztin1990}, but the resulting stability condition, $E/(1+a)<\pi/2$, is not optimal. Here, we show that if inequality (\ref{eq:sc}) holds, all distributions are asymptotically stable. That is, distributed delays stabilize negative feedback loops. 

In section \ref{s:pre}, we set the stage for the main stability results. In section \ref{s:d}, we show the stability for distributions of discrete delays. In section \ref{s:g}, we present the generalization to any distributions and in section \ref{s:b}, we provide illustrative examples.

\section{Preliminary results}\label{s:pre}
Let $\eta$ be a distribution with mean $1$. We consider the family of distributions 
\begin{align}\label{eq:scale}
\eta_E(\tau) = \begin{cases} 
	\eta(\tau/E), & E>0, \\
	H(\tau), & E=0.
\end{cases}
\end{align}
where $H(\tau)$ is the step or heaviside function at 0. The distribution $\eta_E$ has a mean $E \geq 0$. The characteristic equation of equation (\ref{eq:xx}), obtained by making the ansatz $x(t) = \exp(-\lambda t)$, is
\begin{align}\label{eq:ce}
\lambda + a + \int_0^\infty{e^{-\lambda \tau} d\eta_E(\tau)} = 0.
\end{align} 
When condition (\ref{eq:nu}) is satisfied, the distribution $\eta_E$ is asymptotically stable if and only if all roots of the characteristic equation have a negative real part $\text{Re}(\lambda)<0$ \citep{stepan1989}. Condition (\ref{eq:nu}) guarantees that there is no sequence of roots with real parts converging to a non-negative value. The leading roots of the characterisitc equations are therefore well defined. When $E=0$, i.e. when there is no delay, there is only one root, $\lambda < 0$. When $E>0$, the characteristic equation has pure imaginary roots $\lambda = \pm i \omega$ only if $0<\omega<\omega_c = \sqrt{1-a^2}$. Thus, the search for the boundary of stability can be restricted to imaginary parts $\omega \in (0, \omega_c]$ \citep{bernard01}. 

We define 
\begin{align}
C(\omega) = & \int_{0}^{\infty} \cos(\omega \tau) d\eta_E(\tau), \\
S(\omega) = & \int_{0}^{\infty} \sin(\omega \tau) d\eta_E(\tau).
\end{align}
We use a geometric argument to bound the roots of the characteric equation of equation (\ref{eq:xx}) by the roots of the characteristic equation with a discrete delay. More precisely, if the leading roots associated to the discrete delay are a pair of imaginary roots, then all the roots associated to the distribution of delays have negative real parts. We first state a criterion for stability: if $S(\omega)<\omega$ whenever $C(\omega)+a=0$, then 	$f$ is stable. The larger the value of $S(\omega)$, the more ``unstable'' the distribution is. We then show that a distribution of $n$ discrete delays $f_n$ is more stable than an certain distribution with two delays $f^*$, i.e. $S_n(\omega)\leq S^*(\omega)$. We construct $f^*$ and determine that one of the delays of this ``most unstable'' distribution $f^*$ is $\tau^*_1=0$, making it easy to determine its stability using Theorem \ref{th:hayes}. We then generalize for any distribution of delays. 

The next proposition provides a necessary condition for instability. It is a direct consequence of theorem 2.19 in \citep{stepan1989}. We give a short proof for completeness.

\begin{prop}\label{pr:os} If the distribution $\eta_E$ is asymptotically unstable, then there exists $\omega_s \in [0,\omega_c]$ such that $C(\omega_s) + a = 0$ and $S(\omega_s) \geq \omega_s$.
\end{prop}

\begin{proof} Suppose that the distribution $\eta_E$ is asymptotically unstable. The roots of the characteristic equation depend continuously on the parameter $E$ and cannot appear in the right half complex plane. Thus there is a critical value $0<\rho<1$ at which $\eta_{\rho E}$ loses its stability, and this happens when the characteristic equation (\ref{eq:ce}) has a pair of imaginary roots $\lambda = \pm i\omega$ with $0 \leq \omega < \omega_c = \sqrt{1-a^2}$, i.e. through a Hopf point. Splitting the characteristic equation in real and imaginary parts, we have
\begin{align*}
\int_{0}^{\infty} \cos(\omega \tau) d\eta_{\rho E}(\tau) & + a = 0, \\ 
\int_{0}^{\infty} \sin(\omega \tau) d\eta_{\rho E}(\tau) & = \omega.
\end{align*} 
Rewriting in term of $\eta_E$, we obtain
\begin{align*}
\int_{0}^{\infty} \cos(\omega \rho \tau) d\eta_E(\tau) & + a = 0, \\
\int_{0}^{\infty} \sin(\omega \rho \tau) d\eta_E(\tau) & = \omega.
\end{align*} 
Finally, denoting $\omega_s = \rho \omega \leq \omega < \omega_c$, we have
\begin{align*}
\int_{0}^{\infty} \cos(\omega_s \tau) d\eta_E(\tau) & + a = 0, \\
\int_{0}^{\infty} \sin(\omega_s \tau) d\eta_E(\tau) & = \omega = \omega_s/\rho \geq \omega_s.
\end{align*} 
This completes the proof.
\end{proof}

Proposition \ref{pr:os} provides a sufficient condition for stability: 
\begin{cor}\label{th:cor}
The distribution $\eta_E$ is asymptotically stable if (i) $C(\omega)>-a$ for all $\omega \in [0,\omega_c]$ or if (ii) $C(\omega)=-a$, $\omega \in [0,\omega_c]$, implies that $S(\omega)<\omega$.
\end{cor}

Proposition \ref{pr:os} suggests that the scaling $\eta_E=\eta(\tau/E)$ is appropriate for looking at the stability with respect to the mean delay. The mean delay scales linearly, and unstable distributions therefore lose their stability at a smaller values of the mean delay, under this scaling. The condition $S(\omega_s)<\omega_s$ is however not necesssary for stability, as one can find cases where $S(\omega_s)>\omega_s$ even though the distribution is stable. This happens when an unstable distribution switches back to stability as $E$ is further increased (see for instance \citep{boese1989} or \citep{beretta2002} and example \ref{s:s}).

\section{Stability of a distribution of discrete delays}\label{s:d}
We define a density of $n$ discrete delays $\tau_i \geq 0$, and $p_i > 0$, $i=1,...,n$, $n \geq 1$, as 
\begin{align}\label{eq:fn}
 f_n(\tau) & = \sum_{i=1}^{n} p_i \delta(\tau-\tau_i) \\
\intertext{where $\delta(t-\tau_i)$ is a Dirac mass at $\tau_i$, and}
 \sum_{i=1}^{n} p_i \tau_i & = E, \text{ and } \sum_{i=1}^{n} p_i = 1. \nonumber
\end{align} 
In this section, we show that $f_n$ is more stable than a single discrete delay. We do that by observing that among all $n$-delay distributions, $n\geq 2$, that satisfy $C_n(\omega_s)+a=0$ for a fixed value $\omega_s < \omega_c$, the distribution $f^*$ that maximizes $S_n(\omega_s)$,
\begin{equation}
\max_{f_n} \bigl\{ S_n(\omega_s) | C_n(\omega_s)+a=0 \bigr\} = S^*(\omega_s),
\end{equation}
has 2 delays. We show that $S^*(\omega_s)<\omega_s$, implying that all distributions are stable. The following lemma shows how to maximize $S(\omega_s)$ for distributions of two delays.

\begin{lem}\label{th:S} Let $f_2$ be a delay density with mean $E$. Assume in addition that there exists $\omega_s \in (0,\omega_c)$ such that $C(\omega_s) = -a < \cos(\omega_c E)$. Then there exists $\tau_2^*$, $p_1^*$ and $p_2^*$ such that
\begin{align}
\tau_1^* & = 0, \\
p_1^* + p_2^* & = 1, \\
p_1^* + p_2^* \cos(\omega_s \tau_2^*) & =  p_1 \cos(\omega_s \tau_1) + p_2 \cos(\omega_s \tau_2), \\
p_2^* \tau_2^* & = E.
\end{align}
Moreover, there is at most two solutions for $\tau_2^*$ with $\tau_2^*<\pi/\omega_s$. If $\tau_2^*$ is the smallest solution, we have that $\tau_2^* \leq \tau_2$ and 
\begin{align*}
S^*(\omega_s) \equiv \sum_{i=1}^2 p_i^* \sin(\omega_s \tau_i^*) & \geq S(\omega_s).
\end{align*}
\end{lem}

\begin{proof}
To see that there is always a solution, let $c>0$ be the smallest value such that the inequality $\cos(\theta) \geq 1 - c \theta$ is verified for all $\theta$. [$c = 0.725...$ by solving $c=\sin(\theta)$, with $1-\theta \sin(\theta) = \cos(\theta)$.] We have that $1-c \omega E \leq C(\omega) \leq \cos(\omega E)$. Thus, the line $1-d \omega E$ that goes through $C(\omega_s)$ satisfies $d = (1-C(\omega_s))/(\omega_s E) \leq c$, and therefore crosses the curve $\cos(\omega_s E)$ at some points. The smallest solution $\tau_2^*$ is the one such that $1 - d \omega_s \tau_2^* = \cos(\omega_s \tau_2^*)$. This way,
\begin{align*}
p_1^* + p_2^* \cos(\omega_s \tau_2^*) & = 1 - (1-C(\omega_s)) p_2^* \tau_2^*/E, \\
 & = C(\omega_s). 
\end{align*}
These new delay values maximize $S(\omega_s)$ under the constraints that $C(\omega_s)+a=0$ and that the mean remains $E$. That is, we will prove that
\begin{align*}
S^*(\omega_s) \equiv \sum_{i=1}^2 p_i^* \sin(\omega_s \tau_i^*) & \geq \sum_{i=1}^2 p_i \sin(\omega_s \tau_i),
\end{align*}
for all admissible $p_i$, $\tau_i$. Two show that, we recast the problem in a slightly different way. Writing $u=\omega_s \tau_1$, $v=\omega_s \tau_2$ and $T = \omega_s E$, we can express parameters $p_i$ in terms of $(u,v)$:
\begin{equation*}
p_1 = \frac{v-T}{v-u} \quad\text{and}\quad p_2 = \frac{T-u}{v-u},
\end{equation*}
where $u<T<v$. We consider $C$ and $S$ as functions of $(u,v)$. 
\begin{align}
C(u,v) & = \frac{v-T}{v-u} \cos(u) + \frac{T-u}{v-u} \cos(v), \label{eq:cuv} \\
S(u,v) & = \frac{v-T}{v-u} \sin(u)  + \frac{T-u}{v-u} \sin(v). \label{eq:suv}
\end{align}
Equation (\ref{eq:suv}) is to be maximized for $(u,v)$ along the curve $h=\{u,v\}$ implicitly defined by the level curves $C(u,v)=-a$. There are either two solutions in $v$, including multiplicity, of the equation $C(0,v)=-a$ or none, so the curve can be parametrized in a way that $(u(\xi),v(\xi))$ satisfies $(u(0)=0,v(0)=v_{max})$ and $(u(1)=0,v(1)=v_{min})$, with $v_{min} \leq v_{max}$. We claim that $S$ is maximized for $\xi=1$, i.e. $u=0$ and $v=v_{max}$. This is true only if $S(u(\xi),v(\xi))$ is increasing with $\xi$. That is, the curve $h$ must cross the level curves of $S$ upward. It is clear that $S$ is a decreasing function of $v$, for $u$ fixed and an increasing function of $u$, for $v$ fixed. Thus the level curves $S(u,v)=k$ can be expressed as an increasing function $v_{S,k}(u)$ such that
\begin{equation*}
S(u,v_{S,k}(u))=k,
\end{equation*}
when $k$ is in the image of $S$. Likewise, equation (\ref{eq:cuv}) can be solved locally to yield $v_{C,a}(u)$ such that
\begin{equation*}
C(u,v_{C,a}(u))=-a,
\end{equation*}
whenever $-a$ is in the image of $C$. The function $v_{C,a}(u)$ could take two values on the domain of definition. Because $S$ is decreasing in $v$, we choose the lower solution branch for $v_{C,a}(u)$. If, along that lower branch, the slope of $v_{C,a}(u)$ is larger than that of $v_{S,k}(u)$, then as $v$ decreases along the curve $c$, $S$ increases. Therefore, to show that $(0,v_{C,a}(0)=v_{min})$ maximizes $S$, we need to show that
\begin{equation}\label{eq:vCvS}
\frac{d v_{C,a}(u)}{du} > \frac{d v_{S,k}(u)}{du} > 0.
\end{equation}
It is clear that $d v_{S,k}(u)/du>0$. The pointwise derivatives of the level curves at $(u,v)$ are
\begin{align*}
\frac{d v_{C}(u)}{du} & = \frac{v-T}{T-u} \frac{-\cos(u)+\cos(v)+(v-u)\sin(u)}{\cos(u)-\cos(v)-(v-u)\sin(v)}, \\
\frac{d v_{S}(u)}{du} & = \frac{v-T}{T-u} \frac{-\sin(u)+\sin(v)-(v-u)\cos(u)}{\sin(u)-\sin(v)+(v-u)\cos(v)}.
\end{align*}
Because only the lower branch of $v_C$ is considered, we restrict $(u,v)$ where $dv_{C}(u) / du <+\infty$. This is done without loss of generality since $S$ is strictly larger on the lower branch than on the upper branch. Along the lower branch, $v_C(u)<\pi$. Inequality (\ref{eq:vCvS}) then holds if 
\begin{multline*}
(v-u) \bigl[ 2 - 2 \bigl(\cos(u)\cos(v)+\sin(u)\sin(v)\bigr) + (v-u) \bigl(\sin(u)\cos(v)-\cos(u)\sin(v) \bigr) \bigr] > 0.
\end{multline*}
Notice that this inequality does not depend on $T$, which cancels out, nor on $a$, since comparison is made pointwise, for any level curves. The inequality can be simplified and rewritten in terms of $z=v-u>0$,
\begin{equation*}
z \bigl[ 2 - 2 \cos(z) - z \sin(z)\bigl] > 0.
\end{equation*}
It can be verified that this inequality is satisfied for $z \in (0,\pi]$. Therefore, $S$ is maximized when $u=0$ and $v=v_{C,a}(0)=v_{min}$.
\end{proof}

\begin{figure}
\includegraphics[width=0.5\linewidth]{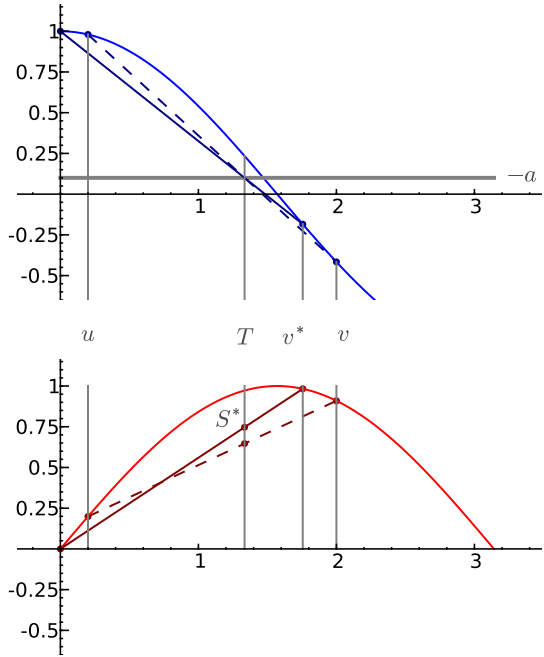}
\caption{How delays are replaced to get an maximal value of $S^*$. In this example, $a=-0.1$. Parameters are $u=0.2$, $v=2$, $p_1=0.37$, $p_2=0.63$ and $T=1.33$. The parameters maximizing $S$ are $u^*=0$, $v^*=1.76$, $p_1^*=0.24$ and $p_2^*=0.76$.}\label{f:2d}
\end{figure}

\begin{theorem}\label{th:n} Let $f_n$ be a density with $n \geq 1$ discrete delays and mean $E$ satisfying inequality (\ref{eq:sc}). The density $f_n$ is asymptotically stable.
\end{theorem}

\begin{proof} Single delay distributions ($n=1$) are asymptotically stable by Theorem \ref{th:hayes}. We first show the case $n=2$.

Consider a density $f_2$, with $\tau_1 < \tau_{2}$. Suppose $C(\omega_s)+a=0$ for a value of $\omega_s<\omega_c$ (if not, Corollary \ref{th:cor} states that $f_2$ is stable). Remark that $-a=C(\omega_s)<\cos(\omega_s E)$. Indeed, from inequality (\ref{eq:sc}) and $\omega_s \leq \omega_c = \sqrt{1-a^2}$, we have  $\cos(\omega_s E) \geq \cos(\omega_c E) > -a$. Replace the two delays by two new delays with new weights: $\tau_1^* = 0$ and $\tau_2^* \geq 0$ the smallest delay such that the following equations are satisfied:
\begin{align}
p_2^* \tau_2^* & = p_1 \tau_1 + p_2 \tau_2, \\
p_1^* + p_2^* \cos(\omega_s \tau_2^*) & =  p_1 \cos(\omega_s \tau_1) + p_2 \cos(\omega_s \tau_2), \\
p_1^* + p_2^* & =  p_1 + p_2 \quad (=1).
\end{align} 
Lemma \ref{th:S} ensures that there always exists a solution when $C(\omega_s) \leq \cos(\omega_s E)$. Additionally, $\tau_2^* \leq \tau_2$ and
\begin{align*}
S^*(\omega_s) \equiv \sum_{i=1}^2 p_i^* \sin(\omega_s \tau_i^*) & \geq \sum_{i=1}^2 p_i \sin(\omega_s \tau_i).
\end{align*}
That is, the new distribution $*$ maximizes the value of $S$. Therefore, if we are able to show that distributions with a zero and a nonzero delay satisfy $S(\omega_s)<\omega_s$, then by Corollary \ref{th:cor}, all distributions with two delays are stable. Consider $f(\tau) = (1-p) \delta(\tau) + p \delta(\tau-r)$. Suppose that there is $\omega_s \leq \omega_c$ such that
\[
C(\omega_s) = 1-p + p \cos(\omega_s r) = -a.
\]
We must show that $S(\omega_s) = p \sin(\omega_s) < \omega_s$. Summing up the squares of the cosine and the sine, we obtain
\[
p^2 = (-a+p-1)^2 + S^2(\omega_s),
\] 
so
\[
S(\omega_s) = \sqrt{p^2 - (-a+p-1)^2}.
\]
By assumption, the mean delay statisfies inequality (\ref{eq:sc}),
\[
pr < \frac{\arccos(-a)}{\sqrt{1-a^2}}.
\]
Thus,
\[
\omega_s = \frac{\arccos \bigl( - (a+1-p) p^{-1} \bigr)}{r} > p \sqrt{1-a^2}\frac{\arccos \bigl( - (a+1-p) p^{-1} \bigr)}{\arccos(-a)}.
\] 
Because $(a+1-p)/p \geq a$ for $p \in (0,1]$ and $a \in (-1,1)$, we have the following inequality
\[
\frac{\arccos(-a)}{\sqrt{1-a^2}} \leq \frac{\arccos \bigl( - (a+1-p) p^{-1} \bigr)}{\sqrt{1-\bigl( (a+1-p) p^{-1} \bigr)^2}}.
\]
Thus, 
\[
S(\omega_s) =  \sqrt{p^2 - (-a+p-1)^2} \leq p \sqrt{1-a^2} \frac{\arccos \bigl( - (a+1-p) p^{-1} \bigr)}{\arccos(-a)} < \omega_s.
\]
This completes the proof for the case $n=2$. 

For distributions $f_n$ with $n>2$ delays, the strategy is also to find a stable distribution that keeps $C(\omega_s)$ constant and increases $S(\omega_s)$, assuming that $C(\omega_s)+a=0$. This requires two steps. In the first one, all pairs of delays $\tau_i<\tau_j$ for which the inequality
\begin{align}\label{eq:belowcos}
\sum_{k \in \{i,j\}} p_k \cos(\omega_s \tau_k)\leq \cos\biggl(\omega_s \sum_{k \in \{i,j\}} p_k \tau_k\biggr),
\end{align}
holds are iteratively replaced by new delays $\tau_i^*=0$ and $\tau_j^*<\tau_j$, as done in Lemma \ref{th:S}. This transformation preserves $E$, $C(\omega_s)$ and increases $S(\omega_s)$. This is repeated until there remains $m<n$ delays with $\tau_i>0$, $i=2,...,m$ such that 
\[
\sum_{k \in \{i,j\}} p_k \cos(\omega_s \tau_k) > \cos\biggl(\omega_s \sum_{k \in \{i,j\}} p_k \tau_k\biggr),
\]  
for $i \neq j \in \{2,...,m\}$, and $\tau_1 = 0$. (The $\tau_i$ are not the same as in the original distribution, the $*$ have been dropped for ease of reading.) The positive delays $\tau_i>0$ satisfy 
\begin{align*}
\sum_{i=2}^m p_i \cos(\omega_s \tau_i) > \cos\Bigl( \omega_s \sum_{i=2}^m p_i \tau_i \Bigr).
\intertext{while, by assumption,}
\sum_{i=1}^m p_i \cos(\omega_s \tau_i) = -a < \cos(\omega_s E).
\end{align*}

The second step is to replace all delays $\tau_i$, $i=2,...,m$ with the single delay $\bar\tau_2 = \sum_{i=2}^m p_i \tau_i$. We now have a 2-delay distribution with $\bar \tau_1=0$ and $\bar \tau_2 > 0$, $\bar p_1 \bar \tau_1 + \bar p_2 \bar \tau_2=E$, $\bar C(\omega_s) \leq C(\omega_s)$ and $\bar S(\omega_s)\geq S(\omega_s)$. Replace $\bar \tau_2$ by the delay $\tau_2^*<\bar \tau_2$ so that $C^*(\omega_s) = -a$, while keeping $E$ constant. Existence of $\tau_2^*$ is shown using the notation from the proof of Lemma \ref{th:S}, and noting that $C(0,v)$ and $S(0,v)$ are both decreasing in $v$. This change of delay has the effect of increasing $S$: $ S^*(\omega_s) \geq \bar S(\omega_s)$. Therefore, we have found a pair of discrete delays $(0, \tau_2^*)$ such that $C^*(\omega_s) = C(\omega_s)$ and $\omega_s > S^*(\omega_s) \geq S(\omega_s)$. By Corollary \ref{th:cor}, $f_n$ is asymptotically stable.
\end{proof} 

\begin{figure}
\includegraphics[width=0.8\linewidth]{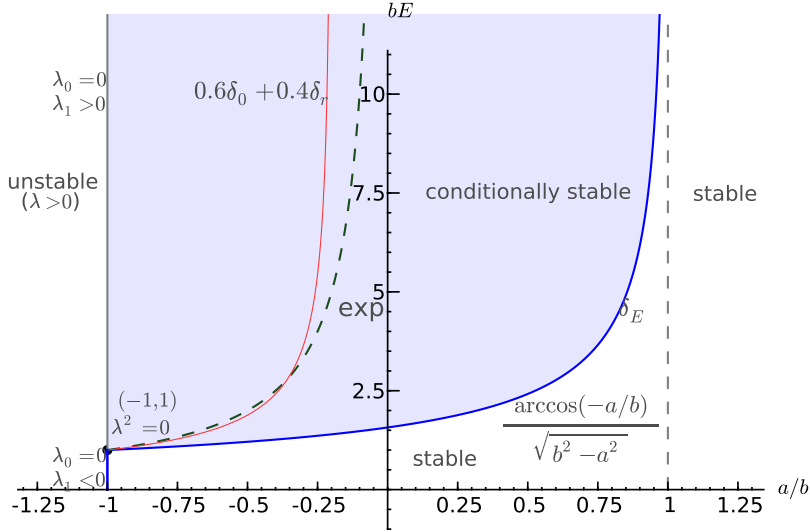}
\caption{Stability chart of distributions of delay in the $(a/b,bE)$ plane. The distribution-independent stability region is to the right of the blue curve. The distribution-dependent stability region is the shaded area. All stability curves leave from the point $(a=-b,E=1/b)$. The signs of the real roots of the characteristic equation $\lambda_0, \lambda_1$ along $a=-b$ are distribution-independent.}\label{f:chart}
\end{figure}

\section{Stability of a general distribution of delays}\label{s:g}
From the stability of distributions of discrete delays to the stability of general distributions of delays, there is a small step. First we need to bound the roots of the characteristic equation for general distributed delays.

\begin{lem}\label{th:mu}
Let $\eta_E$ be a delay distribution with mean $E$ satisfying inequality (\ref{eq:sc}). There exists a sequence $\{\eta_{n,E}\}_{n \geq 1}$ with distribution $\eta_{n,E}$ having $n$ delays, such that $\eta_{n,E}$ converges weakly to $\eta_E$. Then $\lambda$ is a root of the characteristic equation if and only if there exists a sequence of roots $\lambda_n$ for $\eta_{n,E}$ such that $\lim_{n \to \infty} \lambda_n = \lambda$. Let $\{\mu_n\}_{n \geq 1}$ be a sequence of real parts of roots of the characteristic equations. Additionally,
\begin{align*}
 \limsup_{n \to \infty} \mu_n = \mu < 0.
\end{align*}
\end{lem}

\begin{proof}
Consider $\lambda_n = \mu_n + i \omega_n$ a root the characterisitic equation for $\eta_{n,E}$. $E$ satisfies inequality (\ref{eq:sc}), so $\mu_n < 0$. So
\begin{align*}
& \Bigl\lvert \lambda_n + a + \int_0^{\infty} e^{-\lambda_n \tau} d \eta_{E}(\tau) \Bigl\lvert \\
&  = \Bigl\lvert \lambda_n + a + \int_0^{\infty} e^{-\lambda_n \tau} d [\eta_{E}(\tau)-\eta_{n,E}(\tau)] + \int_0^{\infty} e^{-\lambda_n \tau} d \eta_{n,E}(\tau) \Bigr\rvert \\
& = \Bigl\lvert \int_0^{\infty} e^{-\lambda_n \tau} d [\eta_{E}(\tau)-\eta_{n,E}(\tau)] \Bigl\lvert \to 0,
\end{align*}
as $n \to \infty$ by weak convergence. Thus any converging sub-sequence of roots converges to a root for $\eta_E$. The same way, if $\lambda$ is a root for $\eta_E$,  
\begin{align*}
& \Bigl\lvert \lambda + a + \int_0^{\infty} e^{-\lambda \tau} d \eta_{n,E}(\tau) \Bigl\lvert \\
&  = \Bigl\lvert \lambda + a + \int_0^{\infty} e^{-\lambda \tau} d [\eta_{n,E}(\tau)-\eta_{E}(\tau)] + \int_0^{\infty} e^{-\lambda \tau} d \eta_{E}(\tau) \Bigr\rvert \\
& = \Bigl\lvert \int_0^{\infty} e^{-\lambda \tau} d [\eta_{n,E}(\tau)-\eta_{E}(\tau)] \Bigl\lvert \to 0,
\end{align*}
as $n \to \infty$. Convergence is guarantedd by inequality (\ref{eq:nu}). Therefore, each root $\lambda$ lies close to a corresponding root $\lambda_n$.

Denote $\mu = \limsup_{n \to \infty} \mu_n$. Then $\mu$ is the real part of a root of the characteristic equation associated to $\eta_E$. $\mu_n<0$ for all $n$, so $\mu \leq 0$. Suppose $\mu=0$. Without loss of generality, we can assume that all other roots have negative real parts. Then $\eta_E$ is at a Hopf point, i.e. the leading roots of the charateristic equation are pure imaginary. Consider the distribution $\eta_{\bar a,\rho}(\tau) = \eta(\tau/\rho)$ and the associated real parts $\mu_{\bar a, \rho}$, where the subscript $a$ is there to emphasize the dependence of the stability on $a$. Then, by continuity, there exists $(\bar a,\rho)$ in the neighborhood $\varepsilon>0$ of $(a,E)$ for which $\eta_{\bar a, \rho}$ is unstable, i.e. $\mu_{\bar a, \rho}>0$. For sufficiently small $\varepsilon>0$, inequality (\ref{eq:sc}) is still satisfied:
\begin{equation*}
\rho < \frac{\arccos(- \bar a)}{\sqrt{1-\bar a^2}}.
\end{equation*}
Additionally, $\eta_{n,\rho}$ converges weakly to $\eta_{\bar a, \rho}$. However, because $\eta_{\bar a, \rho}$ is unstable, there exists $N>1$ such that $\eta_{n,\bar a, \rho}$ is unstable for all $n>N$, a contradiction to Theorem \ref{th:n}. Therefore $\mu<0$. 
\end{proof}

\begin{theorem}\label{th:main} Let $\eta_E$ be a delay distribution with mean $E$ satisfying inequality (\ref{eq:sc}). The distribution $\eta_E$ is asymptotically stable.
\end{theorem}

\begin{proof}
Consider the sequence of distributions with $n$ delays $\{\eta_{n,E}\}_{n \geq 1}$ where $\eta_{n,E}$ converges weakly to $\eta_E$. By Lemma \ref{th:mu}, the leading roots of the characteristic equation of $\eta_E$ have negative real parts. Therefore $\eta_E$ is asymptotically stable.
\end{proof}

Is there a result similar to Theorem \ref{th:main} for the most stable distribution? That is, is there a mean delay value such that all distributions having a larger mean are unstable? When $a \geq 0$, the answer is no. For instance the exponential distribution with parameter is asymptotically stable for all mean delays, a property called unconditional stability. Other distributions are also unconditionally stable for $a\geq 0$. Anderson has shown that all distributions with smooth enough convex density functions are unconditionally stable \citep{anderson1991}, but densities do not need to be convex to be unconditionally stable. For example, the non-convex density $f(\tau)=0.5 [\delta(\tau)+\delta(\tau-2E)]$ has mean $E$ but is unconditionally stable. However, no distribution is unconditionally stable for all values of $a \in [-1,0)$, although some are for $a\geq a^*$ with $a^*>-1$ (see example below). 

From the results obtained here, we have the most complete picture of the stability of equation (\ref{eq:x}) when the only information about the distribution of delays is the mean (figure \ref{f:chart}). 

\begin{cor}\label{th:stab}
The zero solution of equation (\ref{eq:x}) is asymptotically stable if $a>-b$ and $a \geq |b|$ or if $b>|a|$ and
\begin{equation*}
 E < \frac{\arccos(-a/b)}{\sqrt{b^2-a^2}}.
\end{equation*}
The zero solution of equation \ref{eq:x} may be asymptotically stable (depending on the particular distribution) if $b>|a|$ and
\begin{equation*}
 E \geq \frac{\arccos(-a/b)}{\sqrt{b^2-a^2}}.
\end{equation*}
The zero solution of equation (\ref{eq:x}) is unstable if $a \leq -b$. 
\end{cor}

\section{Boundary of stability}\label{s:b}
The exact boundary of the stability region in the $(a,E)$ plane can be calculated by parametrizing $\bigl(a(u),E(u)\bigr)$. Consider the distribution $\eta$. Then, at the boundary of stability,
\begin{align*}
0 & =  i \omega + a + \int_0^\infty e^{- i \omega \tau} d \eta(\tau/E), \\
& =  i \omega + a + \int_0^\infty e^{- i \omega E \tau} d \eta(\tau),
\intertext{setting $u = E \tau$,}
& =  i \frac{u}{E} + a + \int_0^\infty e^{- i u \tau} d \eta(\tau).
\end{align*} 
Separating the imaginary and the real part, we obtain
\begin{equation}\label{eq:par}
 a(u) = - C(u) \quad \text {and} \quad E(u) = \frac{u}{S(u)},
\end{equation}
for $u \geq 0$. The fact that $u$ depends on $E$ is not a problem: $u \to \infty$ if and only if $E \to \infty$, and $u \to 0$ if and only if $E \to 0$. Equations (\ref{eq:par}) allows systematic exploration of the boundary of stability in the $(a,E)$ plane. 

\subsection{Exponential distribution}
The exponential distribution $f(\tau) =  e^{- \tau}$ has normalized mean 1, and
\begin{align*}
C(u)  = \frac{1}{1+u^2} \quad \text{and} \quad S(u)  = \frac{u}{1+u^2}.
\end{align*}
The stability boundary is given by $E = -1/a$, for $-1 \leq a<0$. Therefore the exponential distribution is not unconditionally stable for $a<0$.

\subsection{Discrete delays}
The exponential distribution is also not the most stable distribution. The density with a zero and a positive delay is $f(\tau)=(1-p) \delta(\tau) + p \delta(\tau-r)$, $p \in (0,1]$. After lumping the zero delay into the undelayed part, the exact stabiltity boundary becomes
\begin{equation*}
E = p r = \frac{\arccos\big(-(a+1-p)p^{-1}\bigr)}{\sqrt{1-\big((a+1-p)p^{-1}\bigr)^2}}
\end{equation*}
This has an asymptote at $a=2p-1$, which can be located anywhere in $(-1,1]$. 

In general, for a distribution with $n$ delays, 
\begin{align*}
a(u)  = -\sum_{i=1}^n p_i \cos(u \tau_i) \quad \text{and} \quad E(u)  = \frac{u}{\sum_{i=1}^n p_i \sin(u \tau_i)}.
\end{align*}
The boundary of the stability region can be formed of many branches, as with a distribution with three delays in figure \ref{f:softkernel}. 

\subsection{Gamma distribution}\label{s:s}
As the mean $E$ is increased, distributions can revert to stability. This is the case with the second order gamma distribution (also called strong kernel) with normalized mean 1, 
\begin{equation}\label{eq:sk}
f(\tau) = 2^2 \tau e^{-2 \tau}.
\end{equation}
We have
\begin{align*}
C(u)  = \frac{1-u^2}{\bigl(1+u^2\bigr)^2}, \quad \text{and} \quad
S(u)  = \frac{2 u}{\bigl(1+u^2\bigr)^2},
\end{align*}
The boundary of stability is given by
\begin{align*}
\bigl(a(u),E(u)\bigr) & = \biggl(\frac{u-1}{(1+u)^2},(1+u)^2\biggr),
\end{align*}
There is a largest value $\hat a = 0.1216$. For large values of $E$, $a \to 0^+$. Therefore the boundary of the stability region is not monotonous; for $a \in (0,\hat a)$, $f$ first becomes unstable and then reverts to stability as the mean is increased (figure \ref{f:softkernel}).

\begin{figure}
\includegraphics[width=0.45\linewidth]{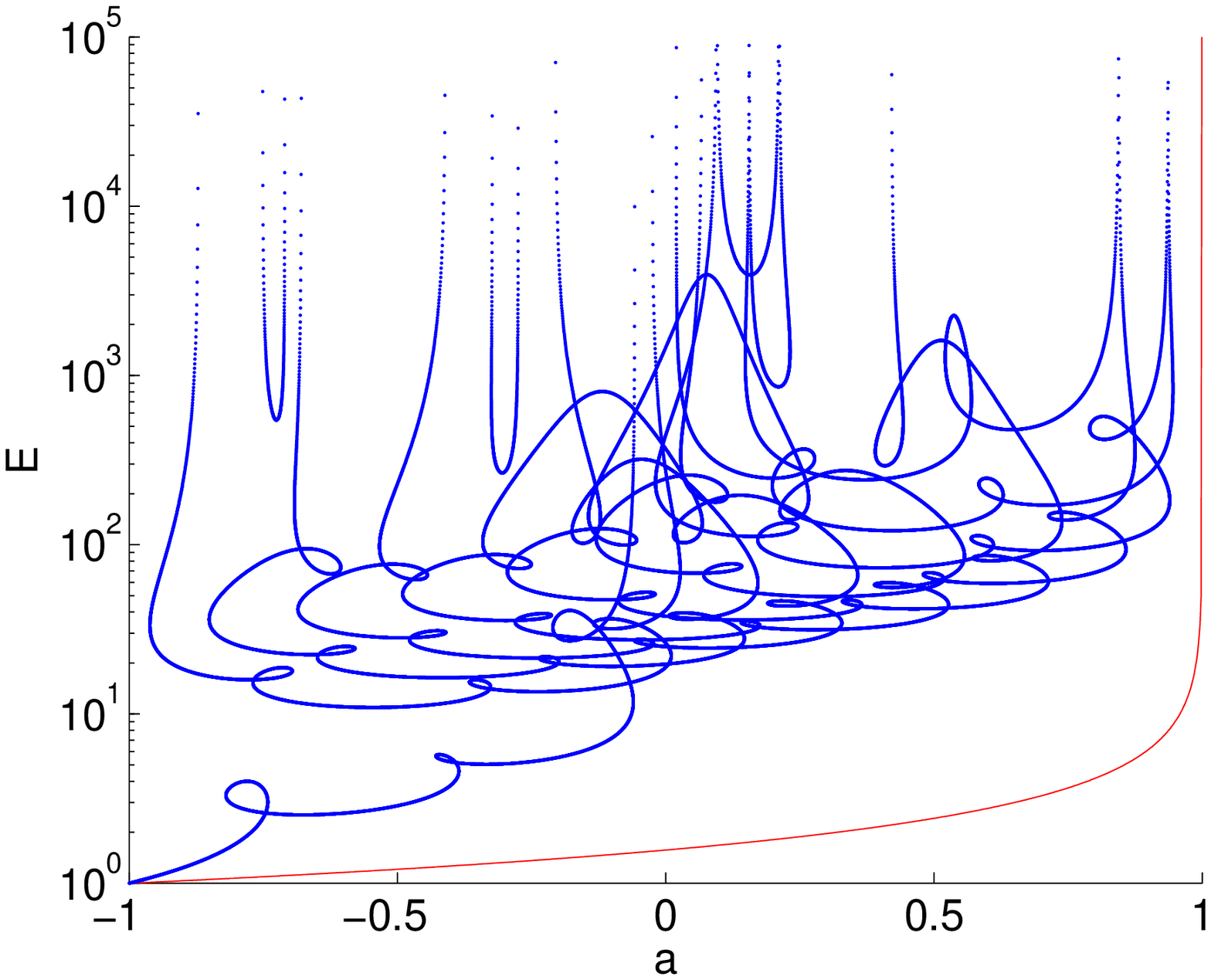}
\includegraphics[width=0.45\linewidth]{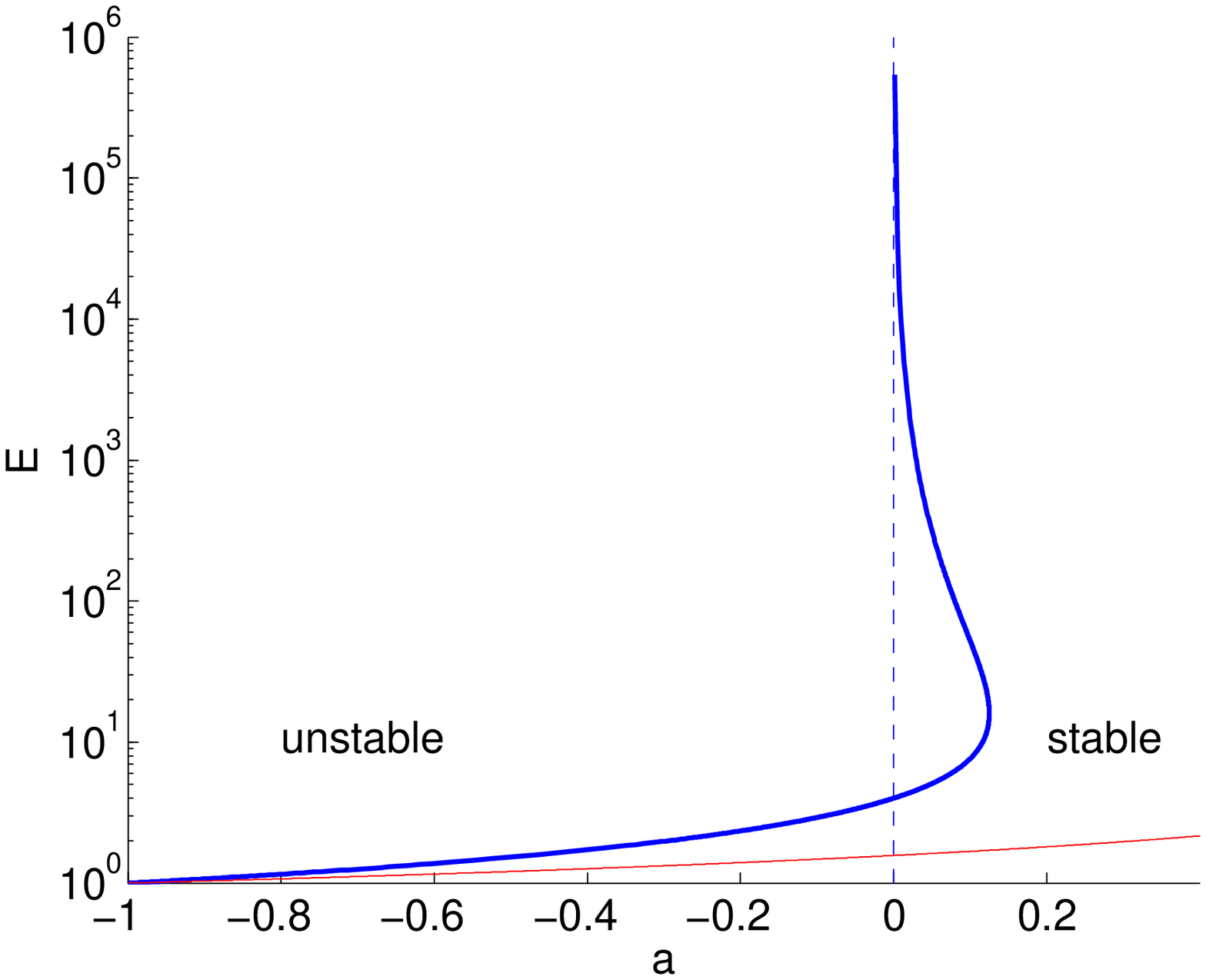}
\caption{(\emph{Left}) Stability chart of the three-delay distribution with $\tau_2 = 16 \tau_1$, $\tau_3 = 96 \tau_1$, $p_1 = 0.51$, $p_2 = 0.39$, $p_3 = 0.1$. (\emph{Right}) Stability chart of the second order gamma distribution, equation (\ref{eq:sk}).}\label{f:softkernel}
\end{figure}


\subsection*{Acknowledgements}
The author thanks Fabien Crauste for helpful discussion. 

\bibliographystyle{plain}
\bibliography{dde}

\begin{thebibliography}{10}

\bibitem{adimy2006}
M.~Adimy, F.~Crauste, A.~Halanay, M.~Neam{\c{t}}u, and D.~Opri{\c{s}}.
\newblock {Stability of limit cycles in a pluripotent stem cell dynamics
  model}.
\newblock {\em Chaos, Solitons and Fractals}, 27(4):1091--1107, 2006.

\bibitem{anderson1991}
R.F.V. Anderson.
\newblock {Geometric and probabilistic stability criteria for delay systems}.
\newblock {\em Mathematical biosciences}, 105(1):81--96, 1991.

\bibitem{atay2003}
F.M. Atay.
\newblock {Distributed delays facilitate amplitude death of coupled
  oscillators}.
\newblock {\em Phys. Rev. Lett.}, 91(9):94101, 2003.

\bibitem{atay2008}
F.M. Atay.
\newblock {Delayed feedback control near Hopf bifurcation}.
\newblock {\em Discrete and Continuous Dynamical Systems B}, 197(205):1, 2008.

\bibitem{beretta2002}
E.~Beretta and Y.~Kuang.
\newblock {Geometric stability switch criteria in delay differential systems
  with delay dependent parameters}.
\newblock {\em SIAM Journal on Mathematical Analysis}, 33(5):1144--1165, 2002.

\bibitem{bernard01}
S.~Bernard, J.~B\'{e}lair, and M.~C. Mackey.
\newblock Sufficient conditions for stability of linear differential equations
  with distributed delay.
\newblock {\em Discrete Contin. Dyn. Syst. Ser. B}, 1:233--256, 2001.

\bibitem{bernard2006b}
S.~Bernard, B.~{\v{C}}ajavec, L.~Pujo-Menjouet, M.C. Mackey, and H.~Herzel.
\newblock {Modelling transcriptional feedback loops: the role of Gro/TLE1 in
  Hes1 oscillations}.
\newblock {\em Philosophical Transactions of the Royal Society A: Mathematical,
  Physical and Engineering Sciences}, 364(1842):1155, 2006.

\bibitem{boese1989}
F.G. Boese.
\newblock {The stability chart for the linearized Cushing equation with a
  discrete delay and Gamma-distributed delays}.
\newblock {\em J. Math. Anal. Appl}, 140:510--536, 1989.

\bibitem{campbell2007}
S.A. Campbell.
\newblock {Time delays in neural systems}.
\newblock {\em Handbook of Brain Connectivity, McIntosh, AR \& Jirsa, VK, ed.
  Springer-Verlag}, 2007.

\bibitem{campbell2009}
S.A. Campbell and R.~Jessop.
\newblock {Approximating the Stability Region for a Differential Equation with
  a Distributed Delay}.
\newblock {\em Mathematical Modelling of Natural Phenomena}, 4(2):1--27, 2009.

\bibitem{eurich2005}
C.W. Eurich, A.~Thiel, and L.~Fahse.
\newblock {Distributed delays stabilize ecological feedback systems}.
\newblock {\em Phys. Rev. Lett.}, 94(15):158104, 2005.

\bibitem{hayes1950}
N.D. Hayes.
\newblock {Roots of the transcendental equation associated with a certain
  difference-differential equation}.
\newblock {\em Journal of the London Mathematical Society}, 1(3):226, 1950.

\bibitem{hutchinson1948}
G.E. Hutchinson.
\newblock {Circular causal systems in ecology}.
\newblock {\em Annals of the New York Academy of Sciences}, 50(4 Teleological
  Mechanisms):221--246, 1948.

\bibitem{kiss2009}
G.~Kiss and B.~Krauskopf.
\newblock {Stability implications of delay distribution for first-order and
  second-order systems}.
\newblock {\em unpublished}, 2009.

\bibitem{krisztin1990}
T.~Krisztin.
\newblock {Stability for functional differential equations and some variational
  problems}.
\newblock {\em Tohoku Math. J}, 42(3):407--417, 1990.

\bibitem{mackey1978}
M.C. Mackey.
\newblock {Unified hypothesis for the origin of aplastic anemia and periodic}.
\newblock {\em Blood}, 51(5):941, 1978.

\bibitem{meyer2008}
U.~Meyer, J.~Shao, S.~Chakrabarty, S.F. Brandt, H.~Luksch, and R.~Wessel.
\newblock {Distributed delays stabilize neural feedback systems}.
\newblock {\em Biological Cybernetics}, 99(1):79--87, 2008.

\bibitem{miyazaki1997}
R.~Miyazaki.
\newblock {Characteristic equation and asymptotic behavior of
  delay-differential equation}.
\newblock {\em Funkcialaj Ekvacioj}, 40(3):481--482, 1997.

\bibitem{monk2003}
N.A.M. Monk.
\newblock {Oscillatory expression of Hes1, p53, and NF-$\kappa$B driven by
  transcriptional time delays}.
\newblock {\em Current Biology}, 13(16):1409--1413, 2003.

\bibitem{rateitschak2007}
K.~Rateitschak and O.~Wolkenhauer.
\newblock {Intracellular delay limits cyclic changes in gene expression}.
\newblock {\em Mathematical Biosciences}, 205(2):163--179, 2007.

\bibitem{sriram2008}
K.~Sriram and S.~Bernard.
\newblock {Complex dynamics in the Oregonator model with linear delayed
  feedback}.
\newblock {\em Chaos: An Interdisciplinary Journal of Nonlinear Science},
  18:023126, 2008.

\bibitem{stepan1989}
G.~St{\'e}p{\'a}n.
\newblock {\em {Retarded dynamical systems: stability and characteristic
  functions}}.
\newblock Longman Scientific \& Technical New York, 1989.

\end{thebibliography}

\end{document}